\documentclass[12pt]{amsart}

\usepackage[margin=1in]{geometry}
\usepackage{amsmath,amsthm,amssymb}
\usepackage{dsfont}
\usepackage{amsmath,amscd}
\usepackage{mathrsfs}
\usepackage{color}
\usepackage{hyperref}
\usepackage{enumitem}
\usepackage{array}
\usepackage{graphicx}
\usepackage{xcolor}
\usepackage{colonequals} 
\usepackage[normalem]{ulem}

\usepackage[colorinlistoftodos]{todonotes}

\usepackage{tikz}
\usepackage{tikz-cd}
\usepackage{tikz-3dplot}

\usetikzlibrary{shapes.geometric, calc}

\hypersetup{
    colorlinks=true, 
    linkcolor=blue, 
    urlcolor=red, 
    linktoc=all 
}

\input xy
\xyoption{all}

\newcommand{\R}{\mathbb{R}}
\newcommand{\C}{\mathbb{C}}
\newcommand{\PP}{\mathbb{P}}
\newcommand{\Z}{\mathbb{Z}}

\newcommand{\Hom}{\mathrm{Hom}}

\newcommand{\sm}[1]{\langle{#1}\rangle}

\numberwithin{equation}{section} 

\def\mooreq{P^3(q)}

\theoremstyle{plain}
\newtheorem{thm}{Theorem}[section]
\newtheorem{prop}[thm]{Proposition}
\newtheorem{lemma}[thm]{Lemma}

\newtheorem{cor}[thm]{Corollary}
\newtheorem{question}[thm]{Question}

\theoremstyle{definition}
\newtheorem{dfn}[thm]{Definition}

\newtheorem{remark}[thm]{Remark}

\begin{document}

\title[The homotopy classification of four-dimensional toric orbifolds]
{The homotopy classification of four-dimensional toric orbifolds}
\author[X. Fu]{Xin Fu}
\address{Department of Mathematics, Ajou University, 
Suwon 16499, Republic of Korea}
\email{xfu87@ajou.ac.kr}

\author[T. So]{Tseleung So}
\address{Department of Mathematics and Statistics, Univeristy of Regina, Regina, SK S4S 0A2, Canada}
\email{tse.leung.so@uregina.ca}

\author[J. Song]{Jongbaek Song}
\address{School of Mathematics, KIAS, Seoul 02455, Republic of Korea}
\email{jongbaek@kias.re.kr}

\subjclass[2010]{Primary 57R18, 55P15; Secondary 55P60}
\keywords{cohomological rigidity, toric orbifold}

\maketitle

\begin{abstract}
Let $X$ be a $4$-dimensional toric orbifold. If $H^3(X)$ has a non-trivial odd primary torsion, then we show that $X$ is homotopy equivalent to the wedge of a Moore space and a CW-complex. As a corollary, given two 4-dimensional toric orbifolds having no 2-torsion in the cohomology, we prove that they have the same homotopy type if and only their integral cohomology rings are isomorphic.
\end{abstract}

\section{Introduction}
One of the central problems in topology is the rigidity question, namely when a weaker equivalence 
between two spaces implies a stronger equivalence between them.
The Freedman's work on the classification of closed oriented simply connected topological $4$-manifolds via the intersection form is a nice example of this type of question. In toric topology, a similar type of question was posed in \cite{MaSu}, which is now called the \emph{cohomologial rigidity problem}, which asks if homeomorphism/diffeomorphism classes of quasitoric manifolds can be classified by their integral cohomology rings. 

Although the problem looks overambitious, it is a sensible question to ask on the following basis. No counter-example has been found since it was formulated. On the contrary, there is a piece of evidence supporting the cohomological rigidity of quasitoric manifolds. Indeed, the classification result in \cite{OrRa} together with the description of the cohomology ring of a quasitoric manifold \cite[Theorem 4.14]{DJ} implies the cohomological rigidity of $4$-dimensional quasitoric manifolds. Besides, many affirmative answers have been proved, for instance certain Bott manifolds \cite{Choi}, generalized Bott manifolds \cite{CMS-tr}, 6-dimensional quasitoric manifolds associated to 3-dimensional Pogorelov polytopes \cite{BEMPP}.

Being a generalized notion of quasitoric manifold, a toric orbifold \cite{DJ} is a 
$2n$-dimensional compact orbifold equipped with a locally standard $T^n$-action whose orbit space is a simple polytope.
It is known that the cohomology rings fail to classify toric orbifolds up to homeomorphism. For instance, there are weighted projective spaces with isomorphic cohomology rings that are not homeomorphic. Therefore, toric orbifolds do not satisfy cohomological rigidity. However, in the above counter-examples, two weighted projective spaces with isomorphic cohomology rings are homotopy equivalent~\cite{BFNR}. 
Hence we take a step back and ask a homotopical version of the cohomogical rigidity: 
\begin{question}\label{ques}
Are two toric orbifolds homotopy equivalent if their integral cohomology rings are isomorphic as graded rings? 
\end{question}

This paper aims to answer this question for certain $4$-dimensional toric orbifolds.
We first study certain CW-complexes which model 4-dimensional toric orbifolds and investigate their homotopy theory. In what follows, $H^\ast(X)$ denotes the cohomology ring with integral coefficients unless otherwise stated, and $P^3(k)$ denotes the $3$-dimensional mod-$k$ Moore space for $k>1$. It is known that $H^3(X)$ is a finite cyclic group. We refer to \cite{Fis, Jor}.  Let~$H^3(X)\cong\Z_m$ with $m=2^sq$ for $q$ odd and $s\geq 0$. When $q>1$, we show that $X$ decomposes into a wedge of $P^3(q)$ and a recognizable space.

\begin{thm}\label{thm_main_1}
Let $X$ be a 4-dimensional toric orbifold such that $H^3(X)\cong \mathbb{Z}_m$. If $m=2^sq$ for an odd integer $q>1$  and  $s\geq 0$, then $X$ is homotopy equivalent to~$\hat{X}\vee P^3(q)$, where $\hat{X}$ is a simply connected 4-dimensional CW-complex with~$H^3(\hat{X})=\Z_{2^s}$ and $H^i(\hat{X})\cong H^i(X)$ for $i\neq 3$.
\end{thm}

If $m$ is odd or equivalently $s=0$, then Theorem~\ref{thm_main_1} implies $X\simeq \hat X\vee P^3(m)$ where $H^3(\hat{X})=0$. As an application,
 we can answer Question~\ref{ques} for certain $4$-dimensional toric orbifolds in the following theorem. 

\begin{thm}\label{thm_rigidity}
Let $X$ and $X'$ be 4-dimensional toric orbifolds such that $H^3(X)$ and $H^3(X')$ have no 2-torsion. Then $X$ is homotopy equivalent to $X'$ if and only if there is a ring isomorphism $H^*(X)\cong H^*(X')$.
\end{thm}

This paper is organized as follows. In Section~\ref{sec_toric_orb}, we review the constructive definition of a~$4$-dimensional toric orbifold $X$. In particular, it is important to see that $X$ is the mapping cone of a map from a lens space to a wedge of $2$-spheres. This phenomenon is motivated by the work of~\cite{BSS} and  can also be understood  in terms of a $\mathbf{q}$-CW complex studied in~\cite{BNSS}. 
In Section~\ref{sec_cohom}, we define a category $\mathscr{C}_{n,m}$ of certain CW-complexes which model 4-dimensional toric orbifolds and study the homotopy theory of~$\mathscr{C}_{n,m}$. 
Sections~\ref{sec_rigidity} aims to give a necessary and sufficient condition for $X\in\mathscr{C}_{n,m}$ to decompose into a wedge of~$P^3(q)$ and a space in~$\mathscr{C}_{n,2^s}$. In Section~\ref{sec_odd_prim_toric_orb}, we study the $p$-local version of the discussion of Section~\ref{sec_rigidity} for some odd prime $p$ and apply this to $4$-dimensional toric orbifolds. Combining the equivalent condition (Proposition~\ref{prop_criterion for splitting}) and the $p$-local decomposition (Proposition~\ref{lemma_p-local main thm}), we finally complete the proofs of Theorem~\ref{thm_main_1} and Theorem~\ref{thm_rigidity} in Section~\ref{sec_proof}.

\section*{acknowledgment}
We started this project during our participation to the Thematic Program on Toric Topology and Polyhedral Products at Fields Institute. We gratefully acknowledge the support of Fields Institute and the organizers of the program. Furthermore, we thank Mikiya Masuda, Taras Panov, Dong Youp Suh and Donald Stanley for discussing the topics, and thank Stephen Theriault for proofreading our draft and giving helpful comments.

The first author was supported by Fields Institute and is supported by the National Research Foundation of Korea funded by the Korean Government (NRF-2019R1A2C2010989),
the second author is supported by Pacific Institute for the Mathematical Sciences (PIMS) Postdoctoral Fellowship and the third author is supported by Basic Science Research Program through the National Research Foundation of Korea (NRF) funded by the Ministry of Education (NRF-2018R1D1A1B07048480) and a KIAS Individual Grant (MG076101) at Korea Institute for Advanced Study.

\section{Toric Orbifolds of Dimension $4$}\label{sec_toric_orb}
We begin with a summary of the constructive definition of a toric orbifold.
For our purpose, we focus on the $4$-dimensional case.
 For more details of toric orbifolds see~\cite[Section 7]{DJ}, \cite[Section 2]{PS} and \cite[Chapter 3, Chapter 10]{CLS}.

Let $P$ be an $(n+2)$-gon  on vertices $v_1 \dots, v_{n+2}$ for some $n\geq 0$. We denote by $E_i$ the edge connecting $v_i$ and $v_{i+1}$ for $i=1, \dots, n+2$, where we take indices modulo $n+2$. To each edge~$E_i$, assign a primitive vector $\xi_i=(a_i,b_i)\in \mathbb{Z}^2$ such that two adjacent vectors $\xi_i$ and~$\xi_{i+1}$ are linearly independent. We often describe this combinatorial data as in Figure \ref{fig_char_data}. 
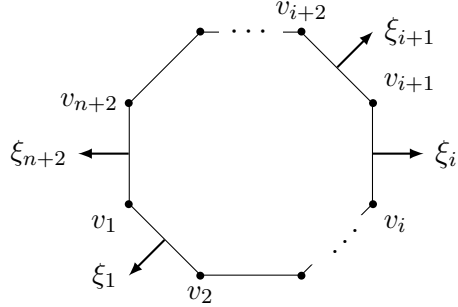
\begin{figure}[ht]
\begin{tikzpicture}
\node[regular polygon, regular polygon sides=8, draw, minimum size=3.5cm](m) at (0,0) {};
\node[fill=white] at (m.side 1) {$\cdots$};
\node[fill=white, rotate=80] at (m.side 6) {$\ddots$};
\fill [black] (m.corner 1) circle (1.5pt) node[above] at (m.corner 1){\small$v_{i+2}$};
\fill [black] (m.corner 3) circle (1.5pt)node[left] at (m.corner 3){\small$v_{n+2}$};
\fill [black] (m.corner 4) circle (1.5pt)node[below left] at (m.corner 4){\small$v_1$};
\fill [black] (m.corner 6) circle (1.5pt) node[below] at (m.corner 5){\small$v_2$};
\fill [black] (m.corner 8) circle (1.5pt) node[above right] at (m.corner 8){\small$v_{i+1}$};
\fill [black] (m.corner 7) circle (1.5pt) node[below right] at (m.corner 7){\small$v_{i}$};
\fill [black] (m.corner 5) circle (1.5pt);
\fill [black] (m.corner 2) circle (1.5pt);
\draw [-latex, thick] (m.side 3) -- ($(m.side 3)!1!90:(m.corner 3)$) node [left] at ($(m.side 3)!1!90:(m.corner 3)$){\small$\xi_{n+2}$};
\draw [-latex, thick] (m.side 4) -- ($(m.side 4)!1!90:(m.corner 4)$) node [left] at ($(m.side 4)!1!90:(m.corner 4)$){\small$\xi_1$};
\draw [-latex, thick] (m.side 8) -- ($(m.side 8)!1!90:(m.corner 8)$) node [right] at ($(m.side 8)!1!90:(m.corner 8)$){$\xi_{i+1}$};
\draw [-latex, thick] (m.side 7) -- ($(m.side 7)!1!90:(m.corner 7)$) node [right] at ($(m.side 7)!1!90:(m.corner 7)$){$\xi_{i}$};
\end{tikzpicture}
\caption{$(n+2)$-gon with primitive vectors on facets.}
\label{fig_char_data}
\end{figure}

Identify $\mathbb{Z}^2$ with $\Hom(S^1, T^2)$. Each $\xi_i$ defines a one-parameter subgroup of $T^2$
$$S^1_{\xi_i}= \{(t^{a_i}, t^{b_i})\in T^2 \mid  t\in S^1\}.$$
Now,  define the following identification space
\begin{equation}\label{eq_identificaion_construction_of_toric_orb}
X= P\times T^2/_\sim
\end{equation}
where $(p,g)\sim(q,h)$ if and only if $p=q$ and 
\begin{enumerate}
\item $gh^{-1}\in S^1_{\xi_i}$ if $p$ is in the relative interior of $E_i$; 
\item $gh^{-1} \in T^2$ if $p=q$ is a vertex of $P$.
\end{enumerate}
Here the torus $T^2$ acts on $X$ by the multiplication on the second factor, which yields the orbit map $\pi \colon X \to P$ by the projection onto the first factor.  

We roughly describe
the orbifold structure on $X$ following the identification~\eqref{eq_identificaion_construction_of_toric_orb}.
First,  there is a standard presentation of $\C^2$ given by a homeomorphism $\mathbb{R}^2_\geq\times T^2/_{\sim_{std}}\cong\mathbb{C}^2$ that maps $[(x,y), (t,s)]$ in $\mathbb{R}^2_\geq \times T^2/_{\sim_{std}}$ to {$(xt,  ys)$~in~$\mathbb{C}^2$}. 
Here the standard identification $\sim_{std}$ 
is defined by $\big((x_1,y_1),g\big)\sim_{std} \big((x_2,y_2),h\big)$ whenever 
\[
(x_1,y_1)=(x_2,y_2) ~\text{in} ~\mathbb{R}^2_\geq
~\text{and}~ 
\begin{cases}
gh^{-1}\in 1\times S^1 &~\text{if}~x_1=0,~y_1\neq 0;\\
gh^{-1}\in S^1\times 1&~\text{if}~x_1\neq 0,~y_1=0;\\
gh^{-1}\in T^2 &~\text{if}~x_1=y_1=0.
\end{cases}\]

Let $U_i$ be a neighborhood of $v_i$ in $P$, which is homeomorphic to $\mathbb{R}^2_\geq$ as a manifold with corners. Let 
$\psi_i$ be a homeomorphism $\R^2_{\geq} \cong U_i$ and let $\rho_i\colon T^2 \twoheadrightarrow T^2$ be an endomorphism of~$T^2$  given by
\begin{equation}\label{eq_aut_rhoi}
\rho_i\colon T^2\to T^2, \quad
(t_1, t_2)\overset{\rho_i}{\longmapsto}(t_1^{a_i}t_2^{a_{i+1}}, t_1^{b_i}t_2^{b_{i+1}}).
\end{equation}
 Since $\xi_i=(a_i, b_i)$ and $\xi_{i+1}=(a_{i+1}, b_{i+1})$ are linearly independent, the kernel~\mbox{$K_i =\ker \rho_i$} is a cyclic subgroup of $T^2$.
Then the map $\psi_i\times \rho_i$ induces a surjection
\begin{equation}\label{eq_chart_of_vi}
\mathbb{C}^2 \cong \mathbb{R}^2_\geq \times T^2/_{\sim_{std}} \xrightarrow{\psi_i\times \rho_i} U_i\times T^2/_\sim.
\end{equation} 

This shows that $U_i\times T^2/_\sim$ is homeomorphic to the quotient $\mathbb{C}^2/K_i$, where $K_i$ acts on $\mathbb{C}^2$ as a subgroup of $T^2$. Hence, 
the map~\eqref{eq_chart_of_vi} forms an orbifold chart around the point $[v_i, g]\in X$. 
The gluing maps among these orbifold charts are determined by the underlying polygon.

A certain cofibration construction of $X$ is studied in \cite{BSS} based on the orbifold structure on $X$.
Pick a vertex $v_i$ of $P$ and $U_i$ is its neighborhood as above. 
Consider a line segment~$\ell_i$ in~$P$ connecting two points  lying in the relative interior of $E_i$ and $E_{i+1}$, respectively. 
The restriction of  identification~\eqref{eq_identificaion_construction_of_toric_orb}  to $\ell_i$ gives rise to a subspace of $X$
\[
L_i =  \ell_i\times T^2/_\sim.
\]
By assuming that the homeomorphism $\psi_i\colon \R^2\to U_i$  sends the arc  $S^1_{\geq}=S^1\cap \R^2_\geq$ to $\ell_i$, the restriction of~\eqref{eq_chart_of_vi}  
to $S^1_\geq \times T^2/_{\sim_{std}}$ 
induces  a homeomorphism  $\big(S^1_{\geq}\times T^2/_{\sim_{std}}\big)/K_i\cong L_i$. Here, we notice that $K_i$ is isomorphic to~$\Z_{m_{i, i+1}}$, where 
\begin{equation}\label{eq_m_ij}
m_{i,j}=|\det \begin{bmatrix}
\xi_i^t &  \xi_{j}^t
\end{bmatrix}|.
\end{equation}
As $S^1_{\geq}\times T^2/_{\sim_{std}}$ is homeomorphic to $S^3$, we conclude that $L_i$ is homeomorphic to $S^3/\Z_{m_{i, i+1}}$ which is $S^3$ if $m_{i,i+1}=1$ and is a lens space otherwise. This description can be found in~\cite[Proposition 2.3]{SaSu} including higher dimensional cases.

Moreover, the subspace $U_i\times T^2/_\sim$  
is homeomorphic to a tubular neighborhood of the cone  on $L_i$.
Let $B$ be the union of all  edges $E_j$ where $j\neq i,i+1$. The subspace $B\times T^2/_\sim$ is homotopic to a wedge of $n$ copies of $2$-spheres  and the subspace $(P-\{v_i\})\times T^2/_\sim$ retracts to $B\times T^2/_\sim$.
As $X$ is  a union of $(P-\{v_i\})\times T^2/_\sim$ and $U_i\times T^2/_\sim$, it implies a homotopy cofibration
\begin{equation}\label{eq_cofib}
L_i \xrightarrow{f_i} \bigvee^{n}_{j=1} S^2 \to X
\end{equation}
where the map $f_i$ is induced by the composition of the inclusion $\iota$ and the retraction $r$
\[
\ell_i\times T^2/_\sim\overset{\iota}\hookrightarrow(P-\{v_i\})\times T^2/_\sim\xrightarrow{r} B\times T^2/_\sim.
\]
See  Figure \ref{fig_cofibraion} for a pictorial illustration of~\eqref{eq_cofib}.
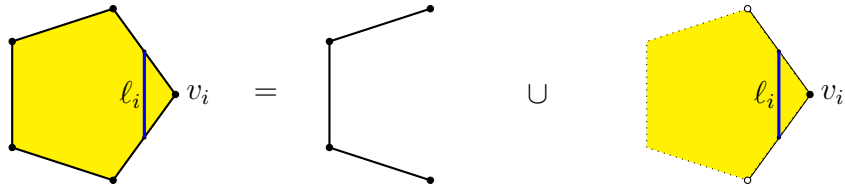
\begin{figure}[ht]
\begin{tikzpicture}[scale=0.6]
\draw[fill=yellow, thick] (0:2)--(72:2)--(144:2)--(216:2)--(288:2)--cycle;

\foreach \vangle in {0, 72,144,216,288}
\draw[fill] (\vangle:2) circle (2pt); 

\node[right] at (0:2) {$v_i$};
\draw[fill] (36:1.62) circle (1pt); 
\draw[fill] (-36:1.62) circle (1pt); 
\draw[blue, very thick] (36:1.62)--(-36:1.62);
\node at (0:1) {$\ell_i$};
\node at (0:4) {$=$};

\begin{scope}[xshift=200]
\foreach \vangle in {72,144,216,288}
\draw[fill] (\vangle:2) circle (2pt); 

\draw[thick] (72:2)--(144:2)--(216:2)--(288:2);
\node at (0:3) {$\cup$};

\end{scope}

\begin{scope}[xshift=400]
\draw[fill=yellow, thick] (72:2)--(0:2)--(288:2);
\draw[fill=yellow, dotted] (0:2)--(72:2)--(144:2)--(216:2)--(288:2)--cycle;
\draw[fill] (0:2) circle (1.5pt); 
\node[right] at (0:2) {$v_i$};
\draw[fill] (36:1.62) circle (1pt); 
\draw[fill] (-36:1.62) circle (1pt); 
\draw[blue, very thick] (36:1.62)--(-36:1.62);
\node at (0:1) {$\ell_i$};

\draw[fill] (0:2) circle (2pt); 
\draw[fill=white] (72:2) circle (2pt); 
\draw[fill=white] (-72:2) circle (2pt); 
\end{scope}
\end{tikzpicture}
\caption{$X=\bigvee_{i=1}^3 S^2 \cup_{f_i} CL_i$.}
\label{fig_cofibraion}
\end{figure}

Applying the cohomology functor to the cofibre sequence~\eqref{eq_cofib} and referring to~\cite[Theorem 1.1]{BNSS}, we can compute the free part of $H^\ast(X)$.
The cohomology of $X$ has been discussed using various tools in the works~\cite[Theorem 2.5.5]{Jor}, \cite[Theorem 2.3]{Fis} and \cite[Corollary 5.1]{KMZ} which can be summarised as follows.
\begin{prop}\label{prop_cohom_of_X}
Let $X$ be a toric orbifold of dimension $4$. Then  we have
\begin{equation}\label{table}
\begin{array}{c|c c c c c c}
i		&0	&1	&2		&3		&4	&\geq5\\
\hline
H^i(X)	&\Z	&0	&\Z^n	&\Z_m	&\Z	&0
\end{array}
\end{equation}
where $m$ is the greatest common divisor of $\{m_{i,j}\mid 1 \leq i< j \leq n+2\}$ for $m_{i,j}$'s defined in~\eqref{eq_m_ij}. We set $\Z_m=0$ if $m=1$.
\end{prop}


\begin{remark}\label{rmk_q_CW}
The way of realising $X$ as a cofibre in~\eqref{eq_cofib} can be understood in a more general framework of a \emph{$\mathbf{q}$-CW complex}.
A  $\mathbf{q}$-CW complex is defined inductively starting from a discrete set $X_0$ of points. Then, $X_{i}$ is defined by the pushout
\[
\begin{tikzcd}
\bigsqcup_\alpha S^{i-1}/K_\alpha \arrow[hook]{r} \arrow{d}{\{\phi_\alpha\}} \arrow[dr, phantom, "\ulcorner", very near end]&\bigsqcup_\alpha e^{i}/K_\alpha \arrow{d}\\
X_{i-1} \arrow{r} & X_i,
\end{tikzcd}
\]
where $e^i$ and $S^{i-1}$ are $i$-dimensional cell and its boundary, respectively, and $K_\alpha$ is a finite group acting linearly on $e_i$. 
Every toric orbifold is a $\mathbf{q}$-CW complex. We refer to \cite{BNSS} for more details. 
\end{remark}

\section{Cohomology of 4-dimensional CW-complexes}\label{sec_cohom}

\subsection{A category of 4-dimensional CW-complexes}

Suppose that $X$ is a simply connected CW-complex  satisfying~\eqref{table}. By~\cite[Proposition 4H.3]{Hat} it is homotopy equivalent to a CW-complex
\begin{equation}\label{eq_mapping_cones}
\left(\underset{i=1}{\overset{n}\bigvee} S^2\vee P^3(m)\right)\cup_f e^4
\end{equation}
where $f\colon S^3\to \bigvee_{i=1}^n  S^2\vee P^3(m)$ is the attaching map of the 4-cell. In this section we study the homotopy theory of CW-complexes in this form.

Define $\mathscr{C}_{n,m}$ to be the full subcategory of $\text{Top}_*$ consisting of mapping cones as in~\eqref{eq_mapping_cones}. Here the orientation of the 4-cell $e^4$ is the induced orientation of the upper hemisphere in~$S^5$. We label the $i^{\text{th}}$ copy of 2-spheres in $\bigvee_{i=1}^n  S^2$ by $S^2_i$ for $1\leq i\leq n$ and  write 
\[
Y=\underset{i=1}{\overset{n}\bigvee} S^2_i\vee P^3(m)
\] 
for short. Let $\mu_i,\nu\in H_2(Y)$ be homology classes representing $S^2_i$ and the 2-cell of $P^3(m)$ respectively. Then, we have 
\begin{equation}\label{eq_H_2(Y)}
H_2(Y)\cong\Z\langle{\mu_1,\ldots,\mu_n}\rangle\oplus\Z_m\langle{\nu}\rangle.
\end{equation}

Let $g\colon Y\to Y$ be a map. Then the induced homology map $g_\ast \colon H_2(Y) \to H_2(Y)$ is given by $g_*(\mu_i)=\sum^n_{j=1}x_{ij}\mu_j+y_i\nu$ and $g_*(\nu)=z\nu$ for some integers $x_{ij}$ and mod-$m$ congruence classes $y_i$ and $z$. Conversely, we have the following lemma. 
\begin{lemma}\label{lemma_self map Y to Y realization}
Given a vector $(y_1,\ldots,y_n,z)\in (\Z_m)^{n+1}$ and an $(n\times n)$-integral matrix
\[
\left(
\begin{array}{c c c}
x_{11}	&\ldots	&x_{1n}\\
\vdots	&\ddots	&\vdots\\
x_{n1}	&\cdots	&x_{nn}
\end{array}
\right)\in \text{Mat}_n(\Z),
\]
there exists a map $g\colon Y\to Y$ such that $g_*(\mu_i)=\sum^n_{j=1}x_{ij}\mu_j+y_i\nu$ and $g_*(\nu)=z\nu$.
\end{lemma}

\begin{proof}
First, consider the string of isomorphisms
\begin{eqnarray*}
 \big[\underset{i=1}{\overset{n}\bigvee} S^2_i,\underset{j=1}{\overset{n}\bigvee} S^2_j\vee P^3(m)\big]
 &\cong& \bigoplus^n_{i=1}\big[S^2_i, \underset{j=1}{\overset{n}\bigvee} S^2_j\vee P^3(m)\big]\\
 &\cong&\bigoplus^n_{i=1}\pi_2\big(\underset{j=1}{\overset{n}\bigvee} S^2_j\vee P^3(m)\big)\\
 &\cong&\bigoplus^n_{i=1}H_2\big(\underset{j=1}{\overset{n}\bigvee} S^2_j\vee P^3(m)\big)\\
 &\cong&\bigoplus^n_{i=1}\left(\bigoplus^n_{j=1}\Z\oplus \Z_m\right)
\end{eqnarray*}
where the third isomorphism is due to Hurewicz Theorem. 
Under these isomorphisms,
take~$g'\colon \bigvee_{i=1}^n  S^2_i\to Y$ to be the map corresponding to
\[
\left(
\begin{array}{c c c}
x_{11}	&\cdots	&x_{1n}\\
\vdots	&\ddots	&\vdots\\
x_{n1}	&\cdots	&x_{nn}
\end{array}
\right)
\oplus
(y_1, \dots, y_n) \in\left(\bigoplus^n_{i=1}\bigoplus^n_{j=1}\Z\right)\oplus\left(\bigoplus^n_{i=1}\Z_m\right).
\]
Then $g'_*(\mu_i)=\sum^n_{j=1}x_{ij}\mu_j+y_i\nu$.

Next, for $z\in\Z_m$ let $g''\colon P^3(m)\to Y$ be the composition
\[
g''\colon P^3(m)\xrightarrow{\underline{z}}P^3(m)\hookrightarrow Y,
\]
where $\underline{z}\colon P^3(m)\to P^3(m)$ is the degree-$z$ map. Let $g:Y\to Y$ be the wedge sum $g=g'\vee g''$. Then $g_*(\mu_i)=\sum^n_{j=1}x_{ij}\mu_j+y_i\nu$ and $g_*(\nu)=z\nu$.
\end{proof}

\subsection{Cellular cup product representation}
Let $C_f\in\mathscr{C}_{n,m}$ be the mapping cone of a map~$f\colon  S^3\to Y$. As the inclusion $Y\hookrightarrow C_f$ induces an isomorphism $H_2(Y)\to H_2(C_f)$, we do not distinguish $\mu_i,\nu\in H_2(Y)$ and their images in $H_2(C_f)$.
 Let~$u_i\in H^2(C_f) \text{ and } e\in H^4(C_f)$
 be cohomology classes dual to $\mu_i$ and the homology class represented by the 4-cell in $C_f$ respectively. Let $v\in H^3(C_f)$ be the Ext image of $\nu$. Then
\begin{equation}\label{eq_cohom_C_f}
\begin{array}{c c c}
H^2(C_f)\cong\Z\sm{u_1,\ldots,u_n},
&H^3(C_f)\cong\Z_m\sm{v},
&H^4(C_f)\cong\Z\sm{e}.
\end{array}
\end{equation}
We call the set $\{u_1,\ldots,u_n,v,e\}$ the \emph{cellular basis of $H^*(C_f)$}. 

With coefficient $\Z_m$, let $\bar{u}_i\in H^2(C_f;\Z_m)$ and~$\bar{e}\in H^4(C_f;\Z_m)$ be the mod-$m$ images of~$u_i$ and~$e$, and let~$\bar{v}\in H^2(C_f;\Z_m)$ be the cohomology class dual to $\nu$.  Then 
\[
\begin{array}{c c c}
H^2(C_f;\Z_m)\cong\Z_m\sm{\bar{u}_1,\ldots,\bar{u}_n,\bar{v}},
&H^3(C_f;\Z_m)\cong\Z_m\sm{\beta(\bar{v})},
&H^4(C_f;\Z_m)\cong\Z_m\sm{\bar{e}},
\end{array}
\]
where $\beta $ is the Bockstein homomorphism.
We call the set~$\{\bar{u}_1,\ldots,\bar{u}_n,\bar{v};\bar{e}\}$ the \emph{mod-$m$ cellular basis} of $H^*(C_f;\Z_m)$.

\begin{dfn}
Let $C_f$ be a mapping cone in $\mathscr{C}_{n,m}$. Then the {\it cellular cup product representation} $M_{cup}(C_f)$ of $C_f$ is $A\in \text{Mat}_n(\Z)$ if $m=1$, and is a triple $(A,\textbf{b},c)\in \text{Mat}_n(\Z)\oplus(\Z_m)^n\oplus\Z_m$ if $m>1$, where
\[
A=\left(
\begin{array}{c c c}
a_{11}	&\cdots	&a_{1n}\\
\vdots	&\ddots	&\vdots\\
a_{n1}	&\cdots	&a_{nn}
\end{array}
\right) 
~\text{and} ~
\mathbf{b}=(b_1, \dots, b_n)\]
are given by $u_i\cup u_j=a_{ij}e, ~\bar{u}_i\cup\bar{v}=b_{i}\bar{e}$ and $\bar{v}\cup\bar{v}=c\bar{e}$.
\end{dfn}
Here $A$ is a symmetric matrix since it is the matrix representation of the bilinear form 
\[(-\cup-)_{\Z}\colon H^2(C_f;\Z)\otimes H^2(C_f;\Z)\to H^4(C_f;\Z)\] with respect to the cellular basis~$\{u_1,\ldots,u_n;e\}$.
 Furthermore, Universal Coefficient Theorem implies~$\bar{u}_i\cup\bar{u}_j=a_{ij}\bar{e}\pmod{m}$. So~$(-\cup-)_{\Z_m}\colon H^2(C_f;\Z_m)\otimes H^2(C_f;\Z_m)\to H^4(C_f;\Z_m)$ can be recovered from $M_{cup}(C_f)$ as well.

\begin{remark}
When $X$ is an oriented compact smooth $4$-manifold, the \textit{intersection form}~$I(X)$ is the bilinear form given by cup products of degree $2$ cohomology classes modulo torsion
\[
I(X)\colon H^2(X)/\text{Tor}\otimes H^2(X)/\text{Tor}\to \Z,\quad x\otimes y\mapsto\langle{x\cup y, [X]}\rangle,
\]
where $[X]\in H_{4}(X)$ is the fundamental class. Although defined in a similar fashion, $I(X)$ and $M_{cup}(X)$ are different. First, $I(X)$ only concerns cup products of free elements in $H^2(X)$ and its matrix representation is a symmetric matrix, while $M_{cup}(X)$ concerns cup products of cohomology with integral and $\Z_m$-coefficients and is a triple consisting of a matrix, a mod-$m$ vector and a mod-$m$ congruence class that record all data. Second, a matrix representation of $I(X)$ depends on the choice of generators of $H^2(X)$, whereas we define $M_{cup}(X)$ using a fixed CW-complex structure of $X$. In the following section, we will discuss the transformation between cellular map representations of two CW-complex structures of the same $X$. It is similar to matrix congruence but is slightly more complicated, as cup products of cohomology with $\Z_m$ coefficient are involved.
\end{remark}

Let $g\colon S^3\to Y$ be another map and let $C_g\in\mathscr{C}_{n,m}$ be its mapping cone. Recall that $f+g$ is the composition
\[
f+g\colon S^3\xrightarrow{\text{comult}}S^3\vee S^3\xrightarrow{f\vee g}Y\vee Y\xrightarrow{\text{fold}}Y.
\]
Denote its mapping cone by $C_{f+g}$.

\begin{lemma}\label{lemma_sum of intersection form}
Let $Y$ be (1) $S^2_1\vee S^2_2$ or (2) $S^2_1\vee P^3(m)$ and let $f,g:S^3\to Y$ be two maps. Then $M_{cup}(C_{f+g})=M_{cup}(C_f)+M_{cup}(C_g)$.
\end{lemma}

\begin{proof}
In the following we only prove Case (2). The proof also works for Case (1) but is simpler. Let 
\begin{itemize}
\item
$\{u,v;e\}$, $\{u_1,v_1;e_1\}$ and $\{u_2,v_2;e_2\}$ be the cellular bases of $H^*(C_{f+g}), H^*(C_f)$ and $H^*(C_g)$, respectively;

\item
$\{\bar{u},\bar{v};\bar{e}\}$, $\{\bar{u}_1,\bar{v}_1;\bar{e}_1\}$ and $\{\bar{u}_2,\bar{v}_2;\bar{e}_2\}$ be the mod-$m$ cellular bases of $H^*(C_{f+g};\Z_m)$, $H^*(C_f;\Z_m)$ and $H^*(C_g;\Z_m)$, respectively;

\item
the cellular cup product representations of $C_{f+g}, C_f$ and $C_g$ be
$M_{cup}(C_{f+g})=(A,\textbf{b},c)$, $M_{cup}(C_{f})=(A_1,\textbf{b}_1,c_1)$ and $M_{cup}(C_{g})=(A_2,\textbf{b}_2,c_2)$, respectively.

\end{itemize}
Here $A,A_1,A_2$ are integers and $\textbf{b},\textbf{b}_1,\textbf{b}_2$ are mod-$m$ congruence classes. We claim that
\[ A=A_1+A_2, ~ \textbf{b}=\textbf{b}_1+\textbf{b}_2 \text{ and } c=c_1+c_2.\]

Consider the mapping cone $C'=Y\cup_{f\vee g}(e^4_1\vee e^4_2)$ of $g\vee h\colon S^3\vee S^3\to Y$. Let
\begin{itemize}
\item
$u'\in H^2(C')$, $e'_1,e'_2\in H^4(C')$ be cohomology classes dual to $S^2$, $e^4_1$ and $e^4_2$;

\item
$\bar{u}'\in H^2(C';\Z_m)$, $\bar{e}'_1,\bar{e}'_2\in H^4(C';\Z_m)$ be the mod-$m$ images of $u,e'_1$ and $e'_2$;

\item
$\bar{v}'\in H^2(C';\Z_m)$ be the cohomology class dual to the 2-cell of $P^3(m)$.
\end{itemize}
Observe that $C_f$ and $C_g$ are subcomplexes of $C'$. Let $\imath_1\colon C_f\to C'$ and $\imath_2\colon C_g\to C'$ be natural inclusions and let $q\colon C_{f+g}\to C'$ be the map collapsing the equatorial disk of the~4-cell in~$C_{f+g}$ to a point. Then
\[
\begin{array}{l l l}
q^*(u')=u,	&q^*(\bar{v}')=\bar{v}	,&q^*(e'_1)=q^*(e'_2)=e,\\[5pt]
\imath^*_j(u')=u_j,	&\imath^*_j(\bar{v}')=\bar{v}_j,	&\imath^*_j(e'_k)=\delta_{jk}e_j
\end{array}
\]
for $j,k\in\{1,2\}$, where $\delta_{jk}$ is the Kronecker symbol. On the one hand, $u'\cup u'=\alpha_1e'_1+\alpha_2e'_2$ for some integers $\alpha_1$ and $\alpha_2$. Now the naturality of cup products implies
\begin{eqnarray*}
\imath^*_j(u'\cup u')&=&\imath^*_j(\alpha_1e'_1+\alpha_2e'_2)\\
\imath^*_j(u')\cup\imath^*_j(u')&=&\alpha_1\imath^*_j(e'_1)+\alpha_2\imath^*_j(e'_2)\\
u_j\cup u_j&=&\alpha_je_j
\end{eqnarray*}
for $j\in\{1,2\}$. So $\alpha_j=A_j$. On the other hand,
\begin{eqnarray*}
u\cup u
&=&q^*(u')\cup q^*(u')\\
&=&q^*(u'\cup u')\\
&=&q^*(A_1e'_1+A_2e'_2)\\
&=&A_1q^*(e'_1)+A_2q^*(e'_2)\\
&=&(A_1+A_2)e.
\end{eqnarray*}
So $A=A_1+A_2$. Similarly we can show $\textbf{b}=\textbf{b}_1+\textbf{b}_2$ and $c=c_1+c_2$. Therefore we have
\[
M_{cup}(C_{f+g})=M_{cup}(C_f)+M_{cup}(C_g).
\]
\end{proof}

\subsection{Cellular map representations} 
Let $f,f'\colon S^3\to Y$ be two maps and $C_f,C_{f'}\in\mathscr{C}_{n,m}$ be their mapping cones. Let
\begin{itemize}
\item
$\{u_1,\ldots,u_n,v,e\}$ and $\{u'_1,\ldots,u'_n,v',e'\}$ be the cellular bases of $H^*(C_f)$ and $H^*(C_{f'})$,
\item
$\{\bar{u}_1,\ldots,\bar{u}_n,\bar{v},\bar{e}\}$ and $\{\bar{u}'_1,\ldots,\bar{u}'_n,\bar{v}',\bar{e}'\}$ be the mod-$m$ cellular bases of $H^*(C_f;\Z_m)$ and $H^*(C_{f'};\Z_m)$.
\end{itemize}


Given a map $\psi\colon C_{f'}\to C_{f}$ and a coefficient ring $R$, let 
\begin{equation*}
\psi^\ast_R\colon H^2(C_f;R)\to H^2(C_{f'};R)
\end{equation*}
be the induced morphism  on the second cohomology with coefficient $R$.

\begin{dfn}\label{defn_cellular map rep}
Let $\psi\colon C_{f'}\to C_f$ be a map. Then the \emph{cellular map representation} $M(\psi)$ of~$\psi$ is $W\in \text{Mat}_n(\Z)$ if $m=1$, and is the triple $(W,\textbf{y},z)\in \text{Mat}_n(\Z)\oplus(\Z_m)^n\oplus\Z_m$ if $m>1$, where
\begin{equation}\label{eq_cellular_mat_rep}
W=\left(\begin{array}{c c c}
x_{11}	&\cdots	&x_{1n}\\
\vdots	&\ddots	&\vdots\\
x_{n1}	&\cdots	&x_{nn}
\end{array}
\right) ~\text{and}~ 
\mathbf{y}=(y_1, \dots , y_n )
\end{equation}
are given by $ {\psi^\ast_\Z}(u_j)=\sum^n_{i=1}x_{ij}u'_i$ and $ {\psi^\ast_{\Z_m}}(\bar{v})=\sum^n_{i=1}y_{i}\bar{u}'_i+z\bar{v}'$.
\end{dfn}

\begin{lemma}\label{lemma_self map induced morphism restriction} 
 {For $R=\Z$ or $\Z_m$, consider $\psi$, $\psi^\ast_R$ and $M(\psi)$ as above.} For $1\leq j\leq n$,
we have $\psi^\ast_{\Z_m}(\bar{u}_j)=\sum^n_{i=1}x_{ij}\bar{u}'_i$. 
Furthermore, if $\psi$ is a homotopy equivalence, then $W$ is an invertible matrix and $z$ is a unit in $\Z_m$.
\end{lemma}

\begin{proof} 
Since $C_f$ and $C_{f'}$ are simply connected, Universal Coefficient Theorem implies that 
\[H^2(C_f {; R})\cong\text{Hom}(H_2(C_f),R),~ H^2(C_{f'} {; R})\cong\text{Hom}(H_2(C_{f'}),R) \text{ and } \psi^\ast_R=\text{Hom}(\psi_\ast, R)\]
 is dual to  $\psi_\ast\colon {H_2(C_{f'})\to H_2(C_f)}$. So $\psi^\ast_{\Z_m}(\bar{u}_j)$ is the mod-$m$ image of $\psi^\ast_{\Z}(u_j)$ and the first part of the lemma follows.

If $\psi$ is a homotopy equivalence, then $W\in\text{Mat}_{n}(\Z)$ and
\[
\left(\begin{array}{c c c c}
\bar{x}_{11}	&\cdots	&\bar{x}_{1n}	&y_1\\
\vdots	&\ddots	&\vdots	&\vdots\\
\bar{x}_{n1}	&\cdots	&\bar{x}_{nn}	&y_n\\
0		&\cdots	&0		&z
\end{array}
\right)\in\text{Mat}_{n+1}(\Z_m)
,\quad\text{where }
\bar{x}_{ij}\equiv x_{ij}\pmod{m}
\]
are invertible matrices. So the second part follows.
\end{proof}

The cellular map representation records the data of $\psi_\Z^\ast$ and $\psi_{\Z_m}^\ast$. The square matrix $W$ in~\eqref{eq_cellular_mat_rep} is the map representation of~$\psi^\ast_{\Z}$ with respect to bases $\{u_1,\ldots,u_n\}$ and $\{u'_1,\ldots,u'_n\}$.  Lemma~\ref{lemma_self map induced morphism restriction} implies that
\[
\left(\begin{array}{c c}
\overline{W}		&\textbf{y}^t\\
\textbf{0}	&z
\end{array}
\right)\in\text{Mat}_{n+1}(\Z_m),
\]
where $\overline{W}$ is the mod-$m$ image of $W$ and $\textbf{0}=(0,\ldots,0)$, is the matrix representation of $\psi^\ast_{\Z_m}$ with respect to bases $\{\bar{u}_1,\ldots,\bar{u}_n,\bar{v}\}$ and $\{\bar{u}'_1,\ldots,\bar{u}'_n,\bar{v}'\}$.

Recall that in linear algebra, matrix representations of a bilinear form $V\otimes V\to \Z$ with respect to different bases of $V$ are congruent to each other. So we have the following lemma.

\begin{lemma}\label{lemma_necessary condition for X simeq X'}
For $C_f,C_{f'}\in\mathscr{C}_{n,m}$, let $M_{cup}(C_f)=(A,\textbf{b},c)$ and $M_{cup}(C_{f'})=(A',\textbf{b}',c')$ be their cellular cup product representations. If there is a homotopy equivalence $\psi \colon  C_{f'}\to C_f$ with $M(\psi)=(W,\textbf{y},z)\in\text{GL}_n(\Z)\oplus(\Z_m)^n\oplus\Z^*_m$, then
\[
\begin{array}{c c c}
A'=W^tAW,
&\textbf{b}'=\textbf{y}AW+z\textbf{b}W,
&c'=\textbf{y}A\textbf{y}^t+2z\textbf{y}\textbf{b}^t+z^2c.
\end{array}
\]\qed
\end{lemma}

In particular, if two maps $f$ and $f'\colon S^3\to Y$ are homotopic, then the matrix cup product representations of their mapping cones are the same. 
\begin{lemma}\label{lemma_htpy_equiv_btw_mapping_cones}
If 
$f$ is homotopic to $f'$, then $M_{cup}(C_f)=M_{cup}(C_{f'})$.
\end{lemma}

\begin{proof}
Take a homotopy $\phi\colon S^3\times I\to Y$ between $f$ and $f'$. It induces a homotopy equivalence~$\Phi\colon C_f\to C_{f'}$ such that its restriction to $Y$ is the identity map. So $M(\Phi)=(I_n,\textbf{0},1)$. Then the lemma follows from Lemma~\ref{lemma_necessary condition for X simeq X'}.
\end{proof}

\begin{lemma}\label{lemma_basis change realization}
Let $C_f\in \mathscr{C}_{n,m}$ and let $(W,\textbf{y},z)$ be a triple in $\text{GL}_n(\Z)\oplus(\Z_m)^n\oplus\Z_m^*$. Then there exist a CW-complex $C_{f'}\in\mathscr{C}_{n,m}$ and a homotopy equivalence $\psi \colon  C_f\to C_{f'}$ such that the cellular map representation $M(\psi)$ is $(W,\textbf{y},z)$.
\end{lemma}

\begin{proof}
Let
\[
W=\left(\begin{array}{c c c}
x_{11}	&\cdots	&x_{1n}\\
\vdots	&\ddots	&\vdots\\
x_{n1}	&\cdots	&x_{nn}
\end{array}
\right)
~\text{and}~
\mathbf{y}=(y_1, \dots, y_n).
\]
By Lemma~\ref{lemma_self map Y to Y realization}, there exists a map $\tilde{\psi}\colon Y\to Y$ such that $\tilde{\psi}_\ast(\mu_i)=\sum^n_{j=1}x_{ij}\mu_j+y_i\nu$ and~$\tilde{\psi}_\ast(\nu)=z\nu$, where $\mu_1, \dots, \mu_n$ and $\nu$ are elements in $H_2(Y)$ as in \eqref{eq_H_2(Y)}. 
Thus, we have $\tilde{\psi}^\ast_\Z(u_i)=\sum^n_{j=1}x_{ji}u_j$ and $\tilde{\psi}^\ast_{\Z_m}(\bar{v})=\sum^n_{i=1}y_i\bar{u}_i+z\bar{v}$.

Let $f'=\tilde{\psi}\circ f$ and let $C_{f'}$ be its mapping cone. Then there is a diagram of cofibration sequences
\[
\xymatrix{
S^3\ar[r]^-{f}\ar@{=}[d]	&Y\ar[r]\ar[d]^-{\tilde{\psi}}	&C_f\ar[d]^-{\psi}\\
S^3\ar[r]^-{f'}				&Y\ar[r]					&C_{f'},
}
\]
where $\psi$ is an induced map. Since $W\in GL_n(\Z)$ and $z\in\Z^*_m$, the middle vertical arrow $\tilde\psi$ induces an isomorphism in homology. By five lemma, $\psi_\ast$ is an isomorphism, which implies that $\psi$ is a homotopy equivalence. Finally, we have $M(\psi)=(W,\textbf{y},z)$ by the construction. 
\end{proof}

\section{The Homotopy theory of complexes in $\mathscr{C}_{n,m}$}\label{sec_rigidity}
\subsection{The $\mathscr{C}_{n,1}$ case}
When $m=1$, the mapping cone $C_f\in \mathscr{C}_{n,1}$ is in the form $\bigvee_{i=1}^n S^2_i \cup_f e^4$ where $f\colon S^3\to\bigvee_{i=1}^n  S_i^2$ is the attaching map of the 4-cell. Hilton--Milnor Theorem (see for instance~\cite[Theorem~7.9.4]{Sel}) implies that $f$ is homotopic to a wedge sum
\begin{equation*}
\sum^n_{i=1}a_i\eta_i+\sum_{1\leq j<k\leq n} a_{jk}\omega_{jk},
\end{equation*}
for some integers $a_i$'s and $a_{jk}$'s. Here $\eta_i$'s and $\omega_{jk}$'s are compositions
\[
\begin{array}{l}
\eta_i\colon S^3\overset{\eta}{\to }S^2_i\hookrightarrow\underset{\ell=1}{\overset{n}{\bigvee}}S^2_\ell\\[8pt]
\omega_{jk}\colon S^3\xrightarrow{[\imath_1,\imath_2]}S^2_j\vee S^2_k\hookrightarrow \underset{l=1}{\overset{n}{\bigvee}} S^2_\ell
\end{array}
\]
of Hopf map $\eta$, Whitehead product $[\imath_1,\imath_2]$ and canonical inclusions of $S^2_i$ and 
$S^2_j\vee S^2_k$ into~${\bigvee}_{\ell=1}^n S^2_\ell$. 
The lemma below shows that the coefficients $a_i$ and $a_{jk}$ are determined by~$M_{cup}(C_f)$.

\begin{lemma}\label{lemma_f_determines_intersection_form}
Let $C_f\in\mathscr{C}_{n,1}$ be the mapping cone of $f\simeq\sum^n_{i=1}a_i\eta_i+\sum_{1\leq j<k\leq n} a_{jk}\omega_{jk}$. If
\begin{equation*}
M_{cup}(C_f)=
\begin{pmatrix}
a'_1	&a'_{12}	&\cdots	&a'_{1n}\\
a'_{12}	&a'_2		&\cdots	&a'_{2n}\\
\vdots	&\vdots		&\ddots	&\vdots\\
a'_{1n}	&a'_{2n}	&\cdots	&a'_n
\end{pmatrix},
\end{equation*}
then $a_i=a'_i$ and $a_{jk}=a'_{jk}$ for all $i,j$ and $k$.
\end{lemma}

\begin{proof}
By Lemma \ref{lemma_htpy_equiv_btw_mapping_cones}, we may assume $f=\sum^n_{i=1}a_i\eta_i+\sum_{1\leq j<k\leq n} a_{jk}\omega_{jk}$.
For $n=2$, let~$C_1, C_2$ and $C_{12}$ be the mapping cones of $a_1\eta_1,a_2\eta_2$ and $a_{12}\omega_{12}$. Then their cellular cup product representations are
\[
M_{cup}(C_1)=\left(
\begin{array}{c c}
a_1	&0\\
0	&0
\end{array}
\right),
\quad
M_{cup}(C_2)=\left(
\begin{array}{c c}
0	&0\\
0	&a_2
\end{array}
\right),
\quad
M_{cup}(C_{12})=\left(
\begin{array}{c c}
0		&a_{12}\\
a_{12}	&0
\end{array}
\right).
\]
By Lemma~\ref{lemma_sum of intersection form}, we have 
\[
M_{cup}(C_f)=\left(
\begin{array}{c c}
a_1		&a_{12}\\
a_{12}	&a_2
\end{array}
\right).
\]
So the lemma holds.

For $n\geq3$, let $\{u_1,\ldots,u_n,e\}$ be the cellular basis of $H^*(C_f)$.  We claim that
\[
u_i\cup u_i=a_i e \quad \text{and} \quad u_j\cup u_k=a_{jk}e,
\]
for each $1\leq i \leq n$ and $1\leq  j<k \leq n$. The composition 
\[
f_{jk}\colon S^3\overset{f}{\to }\underset{l=1}{\overset{n}\bigvee} S^2_l\xrightarrow{\text{pinch}}S^2_j\vee S^2_k
\]
is homotopic to $a_{j}\eta'_1+a_{k}\eta'_2+a_{jk}\omega'_{12}$, where $\eta'_1\colon S^3\overset{\eta}{\to}S^2_j\hookrightarrow S^2_j\vee S^2_k$  and $\eta'_2\colon S^3\overset{\eta}{\to}S^2_k\hookrightarrow S^2_j\vee S^2_k$
are compositions of Hopf map $\eta$ and canonical inclusions and $\omega'_{12}\colon S^3\to S^2_j\vee S^2_k$ is the Whitehead product. Let $C_{jk}$ be the mapping cone of $f_{jk}$. By Lemma~\ref{lemma_htpy_equiv_btw_mapping_cones} and the above argument, we have 
\[
M_{cup}(C_{jk})=\left(
\begin{array}{c c}
a_j		&a_{jk}\\
a_{jk}	&a_k
\end{array}
\right).
\]
Let $\{u'_j,u'_k;e'\}$ be the cellular basis of $H^*(C_{jk})$ and let $\alpha\colon C_f\to C_{jk}$ be the map which pinches all 2-spheres in $C_f$ to the basepoint except for $S^2_j$ and $S^2_k$. Then
\[
\alpha^*(u'_j)=u_j,~ \alpha^*(u'_k)=u_k \text{ and } \alpha^*(e')=e.
\]
By the naturality of cup products, we have
\begin{eqnarray*}
\alpha^*(u'_j)\cup\alpha^*(u'_k)&=&\alpha^*(u'_j\cup u'_k)\\
u_j\cup u_k&=&\alpha^*(a_{jk}e')\\
&=&a_{jk}e.
\end{eqnarray*}
So $a'_{jk}=a_{jk}$. Similarly we can show $a'_i=a_i$. Hence, the lemma follows. 
\end{proof}

Now we classify the homotopy types of CW-complexes in $\mathscr{C}_{n,1}$ by their integral cohomology rings in the next statement.

\begin{prop}\label{lemma_C_n,1 rigidity}
Let $f,f'\colon S^3\to\bigvee_{i=1}^nS^2_i$ be two maps and let $C_f,C_{f'}\in\mathscr{C}_{n,1}$ be their mapping cones. Then $C_f\simeq C_{f'}$ if and only if there is a ring isomorphism $H^*(C_f)\cong H^*(C_{f'})$.
\end{prop}

\begin{proof}
The ``only if'' part is trivial. Assume that $H^*(C_f)\cong H^*(C_{f'})$. Then there is an invertible matrix $W\in\text{GL}_n(\Z)$ such that
\[
W^t\cdot M_{cup}(C_f)\cdot W=M_{cup}(C_{f'}).
\]
By Lemma~\ref{lemma_basis change realization}, there are 
a CW-complex $\tilde{C}\in\mathscr{C}_{n,1}$ and a homotopy equivalence $\psi \colon \tilde{C}\to C_f$ such that $M(\psi)=W$. We claim that $\tilde{C} \simeq C_{f'}$. By Lemma~\ref{lemma_necessary condition for X simeq X'}, we have 
\[
M_{cup}(\tilde{C})=W^t\cdot M_{cup}(C_f)\cdot W=M_{cup}(C_{f'}).
\] 
Let $\tilde{f}$ be the attaching map of the 4-cell in $\tilde{C}$ and let
\[
M_{cup}(\tilde{C})=M_{cup}(C_{f'})=\left(
\begin{array}{c c c c}
a_1		&a_{12}	&\cdots	&a_{1n}\\
a_{12}	&a_2	&\cdots	&a_{2n}\\
\vdots	&\vdots	&\ddots	&\vdots\\
a_{1n}	&a_{2n}	&\cdots	&a_n
\end{array}
\right).
\]
Then Lemma \ref{lemma_f_determines_intersection_form} implies that  $f'$ and $\tilde{f}$ are homotopic to the wedge sum
\[
\sum^n_{i=1}a_i\eta_i+\sum_{1\leq i<j\leq n}a_{ij}\omega_{ij},
\]
which means  $C_{f'}\simeq\tilde{C}$. Hence, we have established the claim. 
\end{proof}

\begin{cor}\label{cor_homotopy_CR_toric_mfd}
Two $4$-dimensional toric orbifolds without torsion in (co)homology are homotopy equivalent if and only if their integral cohomology rings are isomorphic.\qed
\end{cor}

As  4-dimensional quasitoric manifolds always have torsion-free (co)homology, Corollary~\ref{cor_homotopy_CR_toric_mfd} implies that the homotopy types of 4-dimensional quasitoric manifolds are classified by their cohomology rings. As we mentioned in Introduction,  the homeomorphism types of~4-dimensional toric manifolds are cohomologically rigid. One can deduce the conclusion from the topological classification of~4-dimensional smooth manifolds with $T^2$-action studied in \cite{OrRa} together with the cohomology formula \cite[Theorem 4.14]{DJ}. 

 We note that the method in Proposition~\ref{lemma_C_n,1 rigidity}  applies to CW-complexes in $\mathscr{C}_{n,1}$ which are not necessarily manifolds.
We also refer to \cite[Section 5]{DKS} for the computation of the cohomology ring of toric orbifolds considered in Corollary~\ref{cor_homotopy_CR_toric_mfd}.

\subsection{The $\mathscr{C}_{n,m}$ case}

From now on, we assume  $m=2^sq$, where $q>1$ is odd and $s\geq 0$. Recall from~\eqref{eq_cohom_C_f} that  $H^3(C_f)\cong\Z_m$ for~$C_f\in\mathscr{C}_{n,m}$. In this subsection, we discuss the homotopy type of $C_f$. To be more precise, we study a necessary and sufficient condition for a wedge decomposition
\begin{equation}\label{eq_C_f_retractoff}
C_f\simeq \hat{C} \vee P^3(q)
\end{equation}
where $\hat{C}$ is a complex in $\mathscr{C}_{n,2^s}$ so that $H^i(C_f) \cong H^i(\hat{C})$ for $i\neq 3$ and $H^3(\hat{C})\cong\Z_{2^s}$.

\begin{lemma}\label{lemma_nullity test}
Let $q$ be odd and greater than $1$. Consider
\begin{enumerate}[label=(\roman*)]
\item a map 
$g_1\colon S^3\to P^3(q)$ and its mapping cone $C_1$,
\item
a map $g_2\colon S^3\to P^4(q)$ and the mapping cone $C_2$ of the composition
\[
S^3\xrightarrow{g_2}P^4(q)\xrightarrow{[\kappa_1,\kappa_2]}S^2\vee P^3(q),
\]
\end{enumerate}
where $[\kappa_1,\kappa_2]$ is the Whitehead product of inclusions
\[
\kappa_1\colon S^2\to S^2\vee P^3(q) ~\text{and}~ \kappa_2\colon P^3(m)\to S^2\vee P^3(q).
\]
For $i=1$ or $2$, if $H^*(C_i;\Z_m)$ has trivial cup products, then $g_i$ is null homotopic.
\end{lemma}

\begin{proof}
Let $q=p_1^{r_1}\ldots p_{\ell}^{r_{\ell}}$ be a primary factorization of $q$ such that $p_j$'s are different odd primes and all $r_j$'s are at least 1. By Hurewicz Theorem, {$\pi_3(P^4(q))\cong\Z_q$}. By~\cite[Theorem 4]{Porter} and~\cite[Lemma 2.1]{ST19},
\[
\pi_3(P^3(q))\cong\bigoplus^{\ell}_{j=1}\pi_3(P^3(p_j^{r_j}))\cong\bigoplus^{\ell}_{j=1}\Z_{p_j^{r_j}}\cong\Z_q.
\]
It suffices to prove the two cases after localization at $p_j$.

For the $i=1$ case, the lemma is a special case of~\cite[Proposition 4.4]{ST19}. For the $i=2$ case, it can be proved by the argument of~\cite[Proposition 3.2]{ST19} and replacing $P^3(p^t)$ by~$S^2$ and the index $t$ by $\infty$, respectively.
\end{proof}

\begin{lemma}\label{prop_cohmlgy rigid when torsion cup prod trivial}
Let $m=2^sq$, where $q$ is odd and greater than 1. Let $f\colon S^3\to\bigvee_{i=1}^n  S^2_i\vee P^3(m)$ be the attaching map of the 4-cell in $C_f$ and let~$M_{cup}(C_f)=(A,\textbf{b},c)$. If $\textbf{b}\equiv (0,\ldots,0)\pmod{q}$ and $c\equiv 0\pmod{q}$, then there is a CW-complex $\hat{C}\in\mathscr{C}_{n,2^s}$ such that $C_f\simeq\hat{C}\vee P^3(q)$.
\end{lemma}

\begin{proof}
Since $2^s$ and $q$ are coprime, we have $P^3(m)\simeq P^3(2^s)\vee P^3(q)$.
By Hilton--Milnor Theorem, $f$ is homotopic to a wedge sum
\[
f\simeq\sum^n_{i=1}a_i\eta_i+\sum_{1\leq j<k\leq n}a_{jk}\omega_{jk}+\eta'+\sum^n_{i=1}\omega'_{i}+\eta_q+\sum_{i=1}^{n}\omega_{iq}
\]
for some integers $a_i$'s and $a_{jk}$'s. Here $\eta'$, $\omega'_{i}$, $\eta_P$ and $\omega_{iP}$ are compositions
\[
\begin{array}{l}
\eta'\colon S^3\xrightarrow{b'} P^3(2^s)\hookrightarrow\underset{j=1}{\overset{n}\bigvee} S^2_j\vee P^3(2^s)\vee P^3(q)\\[8pt]
\omega'_{i}\colon S^3\xrightarrow{b'_i} P^4(2^s)\xrightarrow{[\kappa'_1,\kappa'_2]} S^2_i\vee P^3(2^s)\hookrightarrow\underset{j=1}{\overset{n}\bigvee} S^2_j\vee P^3(2^s)\vee P^3(q)\\[8pt]
\eta_q\colon S^3\xrightarrow{b_q} P^3(q)\hookrightarrow\underset{j=1}{\overset{n}\bigvee} S^2_j\vee P^3(2^s)\vee P^3(q)\\[8pt]
\omega_{iq}\colon S^3\xrightarrow{b_{iq}} P^4(q)\xrightarrow{[\kappa_1,\kappa_2]} S^2_i\vee P^3(q) \hookrightarrow\underset{j=1}{\overset{n}\bigvee} S^2_j\vee P^3(2^s)\vee P^3(q)\\[8pt]
\end{array}
\]
for some maps $b'$, $b'_i$ and $b_q$, $b_{iq}$. Here 
\[[\kappa'_1,\kappa'_2]:P^4(2^s)\to S^2_i\vee P^3(2^s)\]
is the Whitehead product of inclusions $\kappa'_1:S^2_i\to S^2_i\vee P^3(2^s)$ and $\kappa'_2:P^3(2^s)\to S^2_i\vee P^3(2^s)$, and 
\[ [\kappa_1,\kappa_2]:P^4(q)\to S^2_i\vee P^3(q)\] 
is the Whitehead product of inclusions $\kappa_1:S^2_i\to S^2_i\vee P^3(q)$ and $\kappa_2:P^3(q)\to S^2_i\vee P^3(q)$. If $\eta_q$ and $\omega_{iP}$'s are null homotopic, then $f$ factors through a map $\hat{f}\colon S^3\to\bigvee_{i=1}^n  S^2_i\vee P^3(2^s)$. Let $\hat{C}$ be the mapping cone of $\hat{f}$. Then $C_f\simeq\hat{C}\vee P^3(q)$.

Hence it suffices to show that $\eta_q$ and $\omega_{\ell q}$ are null homotopic for any $\ell$ with $1\leq \ell \leq n$. After localization away from 2, $P^3(2^s)$ becomes contractible and the composition
\[
f_{\ell q}\colon S^3\overset{f}{\to }\bigvee^n_{j=1}S^2_j\vee P^3(q)\xrightarrow{\text{pinch}}S^2_{\ell}\vee P^3(q)
\]
is homotopic to the wedge sum $a_{\ell}\tilde{\eta}_{\ell}+\jmath\circ b_q+[\kappa_1,\kappa_2]\circ b_{\ell q}$, where $\tilde{\eta}$ is the composition of Hopf map~$\eta$ and inclusion $S^2_{\ell}\to S^2_{\ell}\vee P^3(q)$, and $\jmath\colon P^3(q)\to S^2_{\ell}\vee P^3(q)$ is the inclusion.

Consider the diagram of homotopy cofibration sequences
\[
\begin{tikzcd}
\ast\arrow{r}\arrow{d}			&\underset{i\neq \ell}{\overset{n}\bigvee}S^2_i\arrow[equal]{r}\arrow[hook]{d}			&\underset{i\neq \ell}{\overset{n}\bigvee}S^2_i\arrow{d}\\
S^3\arrow{r}{f}\arrow[equal]{d}	&\underset{i=1}{\overset{n}\bigvee} S^2_i\vee P^3(q)\arrow{r}\arrow{d}{\text{pinch}}	&C_f\arrow{d}{\pi}\\
S^3\arrow{r}{f_{\ell q}}			&S^2_\ell\vee P^3(q)\arrow{r}								&C_{\ell q}
\end{tikzcd}
\]
where $C_{\ell q}$ is the mapping cone of $f_{\ell q}$ and $\pi$ is an induced map. Let $\{\bar{u}_1,\ldots,\bar{u}_n,\bar{v};\bar{e}\}$ and~$\{\bar{u}',\bar{v}';\bar{e}'\}$ be the mod-$q$ cellular bases of $H^*(C_f;\Z_{q})$ and $H^*(C_{\ell q};\Z_{q})$. Then
\[
\begin{array}{c c c}
\pi^*(\bar{u}')=\bar{u}_{\ell},
&\pi^*(\bar{v}')=\bar{v},
&\pi^*(\bar{e}')=\bar{e}.
\end{array}
\]
By the hypothesis,  $\pi^*(\bar{u}'\cup\bar{v}')=\bar{u}_{\ell}\cup\bar{v}=0$. Since $\pi^*\colon H^4(C_{\ell q};\Z_{q})\to H^4(C_f;\Z_{q})$ is isomorphic, we have $\bar{u}'\cup\bar{v}'=0$. Similarly we have $\bar{v}'\cup\bar{v}'=0$ and $\bar{u}'\cup\bar{u}'=a_{\ell}\bar{e}$ so that~$M_{cup}(C_{\ell q})=(a_{\ell},0,0)$. Let~$C_1$ and $C_2$ be the mapping cones of $[\kappa_1,\kappa_2]\circ b_{\ell q}$ and~$\jmath\circ b_q$. By Lemma~\ref{lemma_sum of intersection form},
\[
M_{cup}(C_1)=M_{cup}(C_2)=(0,0,0)
\]
so $H^*(C_1;\Z_{q})$ and $H^*(C_2;\Z_{q})$ have trivial cup products. By Lemma~\ref{lemma_nullity test}, $b_{\ell q}$ is null homotopic and so is $\omega_{\ell q}$. Also, notice that $C_2\simeq S^2_{\ell}\vee C'$ where $C'$ is the mapping cone of~$b_q$. So $H^*(C';\Z_{q})$ has trivial cup products. By Lemma~\ref{lemma_nullity test}, $b_q$ is null homotopic and so is $\eta_q$.
\end{proof}

\begin{remark}
In general $\hat{C}$ cannot be further decomposed into a wedge of non-contractible spaces, for example $\hat{C}=\Sigma\R\PP^3$.
\end{remark}

Notice that 
$\hat{C}\vee P^3(q)$ is not contained in $\mathscr{C}_{n,m}$, but it is homotopic to  a mapping cone in~$\mathscr{C}_{n,m}$ as follows. Since $2^s$ and $q$ are coprime to each other, there exist integers $\alpha$ and $\beta$ such that $2^s\alpha+q\beta=1$, where the mod-$q$ congruence class of $\alpha$ and the mod-$2^s$ congruence class of $\beta$ are unique. Identify $\Z_{2^s}\oplus\Z_q$ with $\Z_m$ via the isomorphism
\begin{equation}\label{eq_Z_m_isom}
\rho\colon \Z_{2^s}\oplus\Z_q\to\Z_{m},\quad(x,y)\mapsto q\beta x+2^s\alpha y.
\end{equation}
It induces a homotopy equivalence $\rho:P^3(2^s)\vee P^3(q)\to P^3(m)$. Consider the diagram of homotopy cofibrations
\begin{equation}\label{diagram_wedge equivalent}
\begin{tikzcd}
S^3\arrow{r}{\hat{f}\vee\ast}\arrow[equal]{d}		&\underset{i=1}{\overset{n}\bigvee}S^2_i\vee P^3(2^s)\vee P^3(q)\arrow{r}\arrow{d}{id\vee\rho}			&\hat{C}\vee P^3(q)\arrow{d}{\tilde{\rho}}\\
S^3\arrow{r}{(id\vee\rho)\circ(\hat{f}\vee\ast)}	&\underset{i=1}{\overset{n}\bigvee}S^2_i\vee P^3(m)\arrow{r}	&C'
\end{tikzcd}
\end{equation}
where $C'$ is the mapping cone of $(id\vee\rho)\circ(\hat{f}\vee\ast)$ and $\tilde{\rho}$ is an induced homotopy equivalence. Then $C'\simeq \hat{C}\vee P^3(q)$ via $\tilde{\rho}$.

\begin{lemma}\label{lemma_wedge torsion cup product trivial mod q}
 Let $M_{cup}(C')=(A,\textbf{b},c)$. Then $\textbf{b}\equiv(0,\cdots,0)$ and  $c\equiv0\pmod{q}.$
\end{lemma}
\begin{proof} We prove $b_i\equiv 0\pmod q$.
Let 
\begin{itemize}
\item
$\bar{u}_i\in H^2(\hat{C};\Z_m)$ and $ \bar{e}\in H^4(\hat{C};\Z_m)$ be the mod-$m$ cohomology classes dual to homology classes representing $S^2_i$ and the 4-cell in $\hat{C}$, respectively;
\item
$\mu_i,\,\omega_{2^s},\, \omega_q\in H_2(\hat{C}\vee P^3(q);\Z)$ be the homology classes representing $S_i^2$, the bottom cells of ~$P^3(2^s)$ and $P^3(q)$;
\item
$\bar{w}_{2^s},\,\bar{w}_{q}\in H^2(\hat{C}\vee P^3(q);\Z_m)$ be the cohomology classes such~that
\[
\begin{array}{c c c}
\bar{w}_{2^s}(\omega_{2^s})\equiv q\beta	&\bar{w}_{2^s}(\omega_{q})\equiv 0	&\bar{w}_{2^s}(\mu_i)\equiv 0\\[8pt]
\bar{w}_{q}(\omega_{2^s})\equiv 0		&\bar{w}_{q}(\omega_{q})\equiv 2^s\alpha	&\bar{w}_{q}(\mu_i)\equiv 0
\end{array}\pmod{m}.
\]
\end{itemize}
Denote $\bar{v}=\bar{w}_{2^s}+\bar{w}_q$. Then $\bar{u}_1,\cdots,\bar{u}_n$ and $\bar{v}$ form a basis of $H^2(\hat{C}\vee P^3(q);\Z_m)$.

The right square of (\ref{diagram_wedge equivalent}) implies
\[
\tilde{\rho}^*(\bar{u}'_i)=\bar{u}_i,\quad\tilde{\rho}^*(\bar{v}')=\bar{v}=\bar{w}_{2^s}+\bar{w}_q.
\]

By the naturality of cup products, we have
\begin{eqnarray*}
\tilde{\rho}^*(\bar{u}'_i\cup\bar{v}')&=&\tilde{\rho}^*(b_i\bar{e}')\\
\bar{u}_i\cup(\bar{w}_{2^s}+\bar{w}_q)&=&b_i\bar{e}\\
\bar{u}\cup\bar{w}_{2^s}&=&b_i\bar{e}.
\end{eqnarray*}
Since $\bar{w}_{2^s}$ is a multiple of $q$ and $\bar{e}$ is a generator, $b_i\equiv0\pmod{q}$. Similarly we can show that~$c\equiv0\pmod{q}$.
\end{proof}

\begin{prop}\label{prop_criterion for splitting}
Let $m=2^sq$ as before. For $C_f\in\mathscr{C}_{n,m}$, let $M_{cup}(C_f)=(A,\textbf{b},c)$ where
\[
A=\left(
\begin{array}{c c c}
a_{11}	&\cdots	&a_{1n}\\
\vdots	&\ddots	&\vdots\\
a_{n1}	&\cdots	&a_{nn}
\end{array}
\right)
\text{ and } 
\mathbf{b}=(b_1, \dots, b_n). 
\]
Then $C_f\simeq \hat{C}\vee P^3(q)$ for some $\hat{C}\in\mathscr{C}_{n,2^s}$ if and only if the system of mod-$q$ linear equations 
\begin{equation}\label{eqn_mod m eqn}
\left\{
\begin{array}{c c l}
a_{11}y_1+\cdots+a_{1n}y_n&\equiv&-b_1\\
\vdots & & \\
a_{n1}y_1+\cdots+a_{nn}y_n&\equiv&-b_n\\
b_1y_1+\cdots+b_ny_n&\equiv&-c
\end{array}
\right.
\quad
{\pmod{q}}
\end{equation}
has a solution $(y_1,\ldots,y_n)\in {(\Z_q)^n}$. 
\end{prop}

\begin{proof}
Suppose  $g\colon C'\simeq\hat{C}\vee P^3(q)\to C_f$ is a homotopy equivalence. Let $M(g)=(W,\textbf{y},z)$ and let~$M_{cup}(C')=(A', \mathbf{b}', c')$. By Lemma~\ref{lemma_necessary condition for X simeq X'}, we have 
\begin{equation}\label{eq_basechange_Cf_Chat}
A'=W^tAW, ~\textbf{b}'=\textbf{y}AW+z\textbf{b}W ~\text{ and } ~
c'=\textbf{y}A\textbf{y}^t+2z\textbf{y}\textbf{b}^t+z^2c.
\end{equation}
Lemma~\ref{lemma_wedge torsion cup product trivial mod q} implies $\mathbf{b}'\equiv(0,\cdots,0)$ and $c'\equiv0$ modulo $q$. Since $W$ and $z$ are invertible in $\Z_q$, we can rewrite equations in \eqref{eq_basechange_Cf_Chat} as
\[
~z^{-1}\textbf{y}A\equiv-\textbf{b}  ~\text{ and } ~ z^{-1}\textbf{y}\textbf{b}^t\equiv-c \pmod{q}.
\]
Therefore, $z^{-1}\textbf{y}$ is a solution of \eqref{eqn_mod m eqn}.

Conversely, suppose there is a solution $\mathbf{y}=(y_1,\ldots,y_n)\in(\Z_q)^n$ of \eqref{eqn_mod m eqn}. By Lemma~\ref{lemma_basis change realization}, there exist $C''\in\mathscr{C}_{n,m}$ and a homotopy equivalence $g\colon C''\to C_f$ such that $M(g)=(I,\textbf{y}',1)$ and
$$M_{cup}(C'')=\Big(A,~\mathbf{y}'\bar A +\mathbf{b},~\mathbf{y}'\bar A (\mathbf{y}')^t+2\mathbf{b}(\mathbf{y}')^t+c
\Big),$$
where $\textbf{y}'=(\rho(y_1,0),\cdots,\rho(y_n,0))\in(\Z_m)^n$ for $\rho$ defined in \eqref{eq_Z_m_isom} and $\bar{A}$ is the mod-$m$ image of $A$. Then
\[
\mathbf{y}'\bar A +\mathbf{b}\equiv0~\text{ and }~
\mathbf{y}'\bar A (\mathbf{y}')^t+2\mathbf{b}(\mathbf{y}')^t+c\equiv0\pmod{q}.
\]
Note that Lemma~\ref{prop_cohmlgy rigid when torsion cup prod trivial} implies~$C''\simeq\hat{C}\vee P^3(q)$ for some $\hat{C}\in\mathscr{C}_{n,2^s}$.  Consequently, we have ~$C_f\simeq\hat{C}\vee P^3(q)$. 
\end{proof}

\section{Odd primary local decomposition of toric orbifolds}\label{sec_odd_prim_toric_orb}

Let $X=P\times T^2/_\sim$ be a 4-dimensional toric orbifold associated with the combinatorial data described in Section~\ref{sec_toric_orb}. Since $X$ is simply connected and $H^*(X)$ satisfies~(\ref{table}), \cite[Proposition 4H.3]{Hat} implies that $X$ is in $\mathscr{C}_{n,m}$ up to homotopy. 
Let $m=2^sq$, where $q$ is odd and  $s\geq 0$. In this section, we show that for any odd prime $p$, there is a $p$-local equivalence
\begin{equation}\label{eq_p local equivalence}
X\simeq_{(p)} \hat{X} \vee P^3(q)
\end{equation}
for a CW-complex $\hat{X}$ in $\mathscr{C}_{n,2^s}$   and $P^3(q)$ denotes a point if $q=1$. 

The \textbf{q}-CW complex structure of $X$ with respect to a vertex $v_i$ (see Remark \ref{rmk_q_CW}) implies that $X$ is homotoppy equivalent to the mapping cone of a map
\[
f\colon L_i\to \bigvee^n_{j=1}S^2
\] 
where $L_i$ is the quotient $S^3/\Z_{m_{i,i+1}}$ and $m_{i,i+1}=|\det\begin{bmatrix}\xi^t_{i},\xi^t_{i+1}\end{bmatrix}|$.  Recall that $\Z_{m_{i, i+1}}$ is isomorphic to a subgroup $\ker \rho_i$ of $T^2$, where $\rho_i$ is defined in~\eqref{eq_aut_rhoi}. The $\Z_{m_{i, i+1}}$-action on~$S^3$ is given by the inclusion  $\ker \rho_i\hookrightarrow T^2$ and the standard $T^2$-action on~$S^3$. 
If $m_{i, i+1}=1$, then~$L_i\cong S^3$ and~$X$ is in $\mathscr{C}_{n,1}$. So the equivalence~\eqref{eq_p local equivalence} holds. If $m_{i,i+1}>1$, then~$L_i$ is a lens space $L(m_{i,i+1};k_i)$ for some $k_i$ coprime to $m_{i,i+1}$.

In the following, the \emph{$p$-component} $\nu_p(t)$ of a number $t$ is defined to be the $p$-power $p^r$ such that $p^r$ divides $t$ but $p^{r+1}$ does not. 


\begin{lemma}\label{lemma_localization of suspension L}
For $p$ odd prime, let $\nu_p(m_{i,i+1})=p^r$ and let $L_i=L(m_{i,i+1};k_i)$ be a lens space. Then there is a map $\alpha_p\colon \Sigma L_i\to S^4\vee P^3(p^r)$ that is a $p$-local equivalence.
\end{lemma}

\begin{proof}
Let $m_{i,i+1}=p^r{t}$ where $p$ and $t$ are coprime. Then $P^3(m_{i,i+1})\simeq P^3(p^r)\vee P^3(t)$. Consider the diagram of homotopy cofibration sequences
\[
\begin{tikzcd}
\ast\arrow{r}\arrow{d}			&P^3(t)\arrow[equal]{r}\arrow[hook]{d}	&P^3(t)\arrow{d}\\
S^3\arrow{r}{\phi} \arrow[equal]{d}	&P^3(m_{i,i+1})\arrow{r}\arrow{d}{\text{pinch}}			&\Sigma L_i\arrow{d}{\alpha_p}\\
S^3\arrow{r}{\phi'}			&P^3(p^r)\arrow{r}								&C
\end{tikzcd}
\]
where $\phi$ is the attaching map of the 4-cell in $\Sigma L_i$, $\phi'$ is the composition of $\phi$ and the pinch map, $C$ is the mapping cone of $\phi'$ and $\alpha_p$ is an induced map. The right column induces an exact sequence
\[
\cdots\to \tilde{H}^{i-1}(P^3(t);\Z_{p^r})\to\tilde{H}^i(C;\Z_{p^r})\xrightarrow{\alpha_p^*}\tilde{H}^i(\Sigma L_i;\Z_{p^r})\to\tilde{H}^i(P^3(t);\Z_{p^r})\to\cdots
\]
Since $\tilde{H}^*(P^3(t);\Z_{p^r})=0$, the map  $\alpha_p^*\colon H^*(C;\Z_{p^r})\to H^*(\Sigma L_i;\Z_{p^r})$ is an isomorphism. Moreover,  $H^*(C;\Z_{p^r})$ has trivial cup products because 
$\Sigma L_i$ is a suspension. Now, Lemma~\ref{lemma_nullity test} shows that $\phi'$ is null homotopic, which means $C\simeq S^4\vee P^3(p^r)$. Therefore, we consider~$\alpha_p$ as a map~from $\Sigma L_i$ to $S^4\vee P^3(p^r)$. Since $P^3(t)$ is contractible after $p$-localization, the right column implies that $\alpha_p$ is a $p$-local equivalence.
\end{proof}

\begin{lemma}\label{lemma_net at vertex}
Let $p$ be a prime and let $H^3(X)\cong\Z_m$. Then there exists $i\in\{1,\ldots,n+2\}$ such that $\nu_p(m_{i,i+1})=\nu_p(m)$. 
\end{lemma}

\begin{proof}
By \cite[Corollary 5.1]{KMZ} and \cite[Lemma 3.1]{Fis}, $m=\text{gcd}\{m_{i,j}|1\leq i<j\leq n+2\}$. If $n=1$, the lemma is trivial. So we prove the lemma for $n\geq2$.

Without loss of generality, suppose $\nu_p(m_{1,j})=\nu_p(m)=p^r$ for some $j\in\{3,\cdots,n+1\}$. Let $\xi_1=(a, b)$. Since $a$ and $b$ are coprime, there exist $u$ and $v$ such that
\[
\left(
\begin{array}{c c}
u	&-b\\
v	&a
\end{array}
\right) \in SL_2(\Z).
\]
Changing the basis of $\Z^2$ if necessary, we may assume $\xi_1=(1,0)$.

Let $\xi_2=(x, y)$ and $\xi_{j}=(z,w)$. Then
\[
m_{1,2}=\left|\det\begin{pmatrix}
1	&x\\
0	&y
\end{pmatrix}\right|
=|y|,
\quad
m_{1,j}=\left|\det\begin{pmatrix}
1	&z\\
0	&w
\end{pmatrix}\right|
=|w|.
\]
Write $w=cp^r$ and $y=c'p^s$, where $c$ and $c'$ are integers coprime to $p$ and $s\geq r$. If $s=r$, then $v_p(m_{1,2})=\nu_p(m)$ and consequently the lemma holds. If $s>r$, then
\begin{align*}
m_{2,j}
=|xw-yz|=|cxp^r-c'yp^s| =p^r|cx-c'yp^{s-r}|.
\end{align*}
Since $x$ is coprime to $y$, $x$ is coprime to $p$. So $cx-c'yp^{s-r}$ is coprime to $p$ and $\nu_p(m_{2,j})=p^r$. If $j=3$, then we are done. If not, iterate this argument to $m_{2,j}$, $\xi_3$ and $\xi_j$. Then we can conclude that $\nu_p(m_{j-1,j})=p^r$.
\end{proof}

For any odd prime $p$, let $\nu_p(m)=p^r$. By Lemma~\ref{lemma_net at vertex}, there is an $i\in\{1,\ldots,n+2\}$ such that $\nu_p(m_{i,i+1})=\nu_p(m)=p^r$. Pick the vertex $v_i$ and construct the \textbf{q}-CW-complex structure with respect to $v_i$. Then there is a homotopy cofibration sequence
\[
L_i\overset{f}{\to }\bigvee^n_{j=1}S^2\to  X
\]
with a coaction $c:X\to X\vee\Sigma L_i$. Furthermore, the 3-skeleton of $X$ is $\bigvee^n_{j=1}S^2\vee P^3(m)$ for~$m=2^sq$. 
Let $\hat{X}$ be the quotient $X/P^3(q)$ and let $\phi_p$ be the composition
\begin{equation}\label{eq_phi p}
\phi_p\colon X\overset{c}{\to }X\vee\Sigma L\xrightarrow{\jmath\vee\alpha} \hat{X}\vee P^3(p^r)\vee S^4\xrightarrow{\text{pinch}} \hat{X}\vee P^3(p^r)
\end{equation}
where $\alpha$ is the map in Lemma~\ref{lemma_localization of suspension L} and $\jmath\colon X\to\hat{X}$ is the quotient map.

\begin{prop}[$p$-local version of Main Theorem]\label{lemma_p-local main thm}
Let $p$ be an odd prime. If $\nu_p(m)=p^r$ for some $r\geq 1$, then $\phi_p\colon X\to\hat{X}\vee P^3(p^r)$ is a $p$-local equivalence.
\end{prop}

\begin{proof}
We claim that the map $\phi_p$ in~\eqref{eq_phi p} induces an isomorphism on $\Z_{(p)}$-cohomology
 \begin{equation}\label{eq_phi p  in cohomology}
 \phi^*_p\colon H^*(\hat{X}\vee P^3(p^r);\Z_{(p)})\to H^*(X;\Z_{(p)})
 \end{equation}
 where $\Z_{(p)}$ is the ring of $p$-local  integers.

The cofibration sequence $\mooreq\hookrightarrow X\overset{\jmath}{\to}\hat{X}$ induces an exact sequence
\[
\cdots\to\tilde{H}^{i-1}(\mooreq)\to \tilde{H}^i(\hat{X})\overset{\jmath^*}{\to}\tilde{H}^i(X)\to\tilde{H}^i(\mooreq)\to\tilde{H}^{i+1}(\hat{X})\to\cdots 
\]
For $i\neq 3$, since $\tilde{H}^i(P^3(p^r))=0$ and $\jmath^\ast\colon H^i(\hat{X})\to H^i(X)$ is an isomorphism, the map~\eqref{eq_phi p in cohomology} is an isomorphism.

Next, consider the cofibration sequence $L_i\overset{f}{\to}\bigvee_{i=1}^n  S^2\overset{\imath}{\to}X\overset{\delta}{\to}\Sigma L_i$, where $\imath$ is the inclusion and $\delta$ is the coboundary map. It induces an exact sequence
\[
\cdots\to\tilde{H}^i(\Sigma L_i;\Z_{(p)})\xrightarrow{\delta^*}\tilde{H}^i(X;\Z_{(p)})\to\tilde{H}^i(\underset{i=1}{\overset{n}\bigvee} S^2;\Z_{(p)})\to\tilde{H}^{i+1}(\Sigma L;\Z_{(p)})\to\cdots
\]
Since $\imath^*\colon H^2(X;\Z_{(p)})\to H^2(\bigvee^n_{i=1}S^2_i;\Z_{(p)})$ is an isomorphism, $\delta^*\colon H^3(\Sigma L;\Z_{(p)})\to H^3(X;\Z_{(p)})$ is an isomorphism. Consider the following commutative diagram
\[
\xymatrix{
X\ar[r]^-{c}\ar[d]^-{=}  &X\vee\Sigma L_i\ar[r]^-{{\color{red}\jmath}\vee\alpha_p}\ar[d]	&\hat{X}\vee P^3(p^r)\vee S^4\ar[r]\ar[d]	&\hat{X}\vee P^3(p^r)\ar[d]\\
X\ar[r]^-{\delta} 		&\Sigma L_i\ar[r]^-{\alpha_p}					&P^3(p^r)\vee S^4\ar[r]	&P^3(p^r)
}
\]
where the composite in the upper row is $\phi_p$ in~\eqref{eq_phi p} and the unnamed arrows are pinch maps. The left square commutes due to the property of the coaction map. By Lemma~\ref{lemma_localization of suspension L}, the map~$\alpha_p^*\colon H^3(P^3(p^r)\vee S^4;\Z_{(p)})\to H^3(\Sigma L;\Z_{(p)})$ is isomorphic, so the composite in the lower row induces an isomorphism $H^3(P^3(p^r);\Z_{(p)})\to H^3(X;\Z_{(p)})$.
 Since $H^3(\hat{X};\Z_{(p)})=0$, the map~\eqref{eq_phi p in cohomology}  is an isomorphism for $i=3$.
 
Therefore $\phi^*_p\colon H^*(\hat{X}\vee P^3(p^r);\Z_{(p)})\to H^*(X;\Z_{(p)})$ is an isomorphism and $\phi_p$ is a $p$-local equivalence.
\end{proof}

\begin{lemma}\label{lemma_p-local existence of eqn}
Let $X$ be a 4-dimensional toric orbifold with $H^3(X)\cong\Z_m$, and let $\nu_p(m)=p^r$ for some odd prime $p$ and $r\geq 1$. If $M_{cup}(X)=(A,\textbf{b},c)$ where
\[
A=\left(
\begin{array}{c c c}
a_{11}	&\cdots	&a_{1n}\\
\vdots	&\ddots	&\vdots\\
a_{n1}	&\ldots	&a_{nn}
\end{array}
\right)
\text{ and }
\mathbf{b}=(b_1, \dots, b_n), 
\]
then the system of mod-$p^r$ linear equations
\[
\left\{
\begin{array}{c c l}
a_{11}y_1+\ldots+a_{1n}y_n&\equiv&-b_1\\
\vdots && \\
a_{n1}y_1+\ldots+a_{nn}y_n&\equiv&-b_n\\
b_1y_1+\ldots+b_ny_n&\equiv&-c
\end{array}
\right.
\quad
\pmod{p^r}
\]
has a solution in $(\Z_{p^r})^n$.
\end{lemma}

\begin{proof}
By Proposition~\ref{lemma_p-local main thm} there is a map $\phi_p\colon X\to\hat{X}\vee P^3(p^r)$ that becomes a homotopy equivalence after localized at $p$, where $\hat{X}\in\mathscr{C}_{n,2^s}$ is the quotient $X/P^3(q)$. Let
$$
M(\phi_p)=(W,\textbf{y},z)\in\text{Mat}_n(\Z)\oplus(\Z_m)^n\oplus\Z_m
$$
be the cellular map representation of $\phi_p$. After $p$-localization, $W$ is an invertible matrix and~$z$ is a unit. The lemma follows from Proposition~\ref{prop_criterion for splitting}.
\end{proof}

\section{Proof of the Main Theorems}\label{sec_proof}
\begin{proof}[Proof of Theorem~\ref{thm_main_1}]
Let $q=p_1^{r_1}\ldots p_k^{r_k}$ be the primary factorization where $p_i$'s are different odd primes and $r_i\geq 1$. For each prime $p_i$, Lemma~\ref{lemma_p-local existence of eqn} implies that the mod-$p_i^{r_i}$ version of~\eqref{eqn_mod m eqn} has a solution. By Chinese Remainder Theorem, they give a mod-$q$ solution for~\eqref{eqn_mod m eqn}.  By Proposition~\ref{prop_criterion for splitting}, $X$ is homotopy equivalent to $\hat{X}\vee P^3(q)$.
\end{proof}



\begin{proof}[Proof of Theorem~\ref{thm_rigidity}] The ``only if'' part is trivial. To prove the ``if'' part, let $X$ and $X'$ be~4-dimensional toric orbifolds such that $H^3(X)\cong\Z_m$ and $H^3(X')\cong\Z_{m'}$ for $m$ and $m'$ odd. The hypothesis implies that $H^3(X)\cong H^3(X')$, hence we have $m=m'$. By Theorem~\ref{thm_main_1},  we have $X\simeq\hat{X}\vee P^3(m)$ and $X'\simeq\hat{X}'\vee P^3(m)$ for some $\hat{X},\hat{X}'\in\mathscr{C}_{n,1}$. Since $H^i(X)\cong H^i(\hat{X})$ and $H^i(X')\cong H^i(\hat{X}')$ for $i\neq3$,  we have  $H^*(\hat{X})\cong H^*(\hat{X}')$. Then Proposition~\ref{lemma_C_n,1 rigidity} implies that $\hat{X}\simeq\hat{X}'$, which yields $X\simeq X'$.
\end{proof}

\newcommand{\etalchar}[1]{$^{#1}$}

\end{document}